\documentclass[letterpaper, 10 pt, conference]{ieeeconf}

\IEEEoverridecommandlockouts  

\usepackage[noadjust]{cite}

\newtheorem{theorem}{Theorem}

\newtheorem{assumption}{Assumption}

\usepackage{macros}

\usepackage[utf8]{inputenc}

\title{\bfseries Optimization in Open Networks via Dual Averaging}

\author{Yu-Guan Hsieh$^{1}$, Franck Iutzeler$^{1}$, Jérôme Malick$^{2}$, and Panayotis Mertikopoulos$^{2,3}$
\thanks{This work has been partially supported by MIAI@Grenoble Alpes (ANR-19-P3IA-0003)}
\thanks{$^{1}$ Univ. Grenoble Alpes, LJK, 38000 Grenoble, France {\tt\small firstname.lastname@univ-grenoble-alpes.fr}}%
\thanks{$^{2}$ CNRS, Univ. Grenoble Alpes, 38000 Grenoble, France {\tt\small firstname.lastname@univ-grenoble-alpes.fr}}%
\thanks{$^{3}$ Criteo AI Lab}%
}

\begin{document}

\maketitle

\thispagestyle{empty}
\pagestyle{empty}

\begin{abstract}
In networks of autonomous agents (\eg fleets of vehicles, scattered sensors), the problem of minimizing the sum of the agents' local functions has received a lot of interest. 
We tackle here this distributed optimization problem in the case of open networks when agents can join and leave the network at any time.
Leveraging recent online optimization techniques, we propose and analyze the convergence of a decentralized asynchronous optimization method for open networks.
\end{abstract}

\section{Introduction}
\label{sec:introduction}
Multi-agent systems are a powerful modeling framework for the analysis of signal processing or machine learning over sensor networks, fleets of autonomous vehicles, opi\-nion dynamics, etc. 
This framework calls for decentralized optimization methods where the agents seek to minimize the sum of the individual functions by exchanging information through a communication graph, without the help of a central authority. Indeed, such methods are key to perform distributed computing, signal processing, or learning from scattered sources in communication-constrained or large-scale environments; see e.g \cite{tsitsiklis1984problems,boyd2006randomized,dimakis2010gossip,lian2018asynchronous} and references therein. 

In this regard, decentralized optimization methods have been extensively studied in the literature. Depending on the agents' computing abilities and tasks at hand, several types of algorithms were considered. On the one hand, gradient-based methods such as decentralized gradient descent \cite{nedic2009distributed,bianchi2012convergence,jakovetic2014fast}, and decentralized dual averaging \cite{duchi2011dual,tsianos2012push,lee2017stochastic,colin2016gossip} update the local variables using gradients of the agents' objective functions (see also \cite{NOR18} for a recent review).
On the other hand, splitting methods such as the alternating direction method of multipliers (ADMM) \cite{boyd2011distributed,iutzeler2013asynchronous,wei20131} imply that each agents has to minimize (a regularized version of) its local function, which can be too demanding depending on the application. 

Despite the abundance of literature on this topic, most of them assume a network of fixed composition, that is, the agents that participate in the process always remain the same.
On the contrary, this work focuses on the case of an \emph{open network} where the agents can join and leave the system at any moment.
This can happen in numerous situations, \eg 

\begin{itemize}[leftmargin=5mm,itemsep=1mm]
    \item When a cluster of servers is used to train a machine learning task, a node may leave the network due to a system failure or simply because the resource is acquired by another job.
    New nodes can also be deployed to accelerate the training process or process additional data.
    \item We can also think of the case of volunteer computing where volunteers provide computing resources for a distributed task. The network is naturally open since a device is only involved when the volunteer desires to participate.
    \item In multi-vehicle coordination, the set of vehicles that are considered by the algorithm can evolve with time.
\end{itemize}

Due to the dynamic nature of these open multi-agent systems, several works have studied the stability of consensus algorithm over the mean, maximum, or median of the agents values \cite{Franceschelli2018proportional,Dashti2019dynamic,franceschelli2020stability,de2019lower,de2020open}.
Among the very few works that tackle the problem of decentralized optimization over open networks,  \cite{hendrickx2020stability} showed that decentralized gradient descent is stable when agents/functions change over time if their objectives are sufficiently smooth.

Distributed algorithms also need to cope with asynchronous communications (\ie the agents do not synchronize to communicate between optimization steps) with delays (\ie there may be some gap between sending and receiving times).
Such capacity is an important feature for scalability and flexibility.
The study of this aspect was thus concomitant to the development of different decentralized optimization methods; see \eg \cite{tsianos2012consensus,wu2017decentralized,zhou2018distributed}. 

In this paper, we focus on the particular case of asynchronous open networks where i) the agents can communicate with each other asynchronously following a time-varying communication graph; ii) the exchanges and local processing incur delays; and iii) the agents may join and leave the system for arbitrary periods of time. While the first two points are relatively well studied in the literature as mentioned above (see also \cite{BKR19} for an online distributed algorithm with local processing delays), the last point tremendously complicates the analysis since the system may completely change from one step to another, see \eg \cite{de2020fundamental} and references therein.

To address this challenge, we develop the idea that (offline) optimization problems over open networks can be efficiently handled by tools from \emph{online optimization}.
Using this viewpoint and building on recent results on dual averaging for online learning with delays \cite{hsieh2020multi},
we introduce \acs{DAERON} (\acl{DAERON}), a method for optimization over open networks that allows for asynchronous communications.
\acused{DAERON}
On the theoretical side, we study the algorithm's performance 
with respect to the average (over both time and agents) of the agents' functions.
We then provide numerical simulations on a decentralized regression problem to illustrate the potential of our method.

The rest of the paper is organized as follows.
In \cref{sec:framework}, we formulate the open multi-agent optimization problem mathematically and define the corresponding performance measures.
The main algorithm is described in \cref{sec:DAERON} and analyzed in \cref{sec:analysis}.
\cref{sec:experiments} is dedicated to numerical experiments.
Finally, \cref{sec:conclusion} concludes the paper and provides several clues for future research.

\section{An Open Network of Computing Agents}
\label{sec:framework}
\subsection{Model}

We consider a (possibly infinite) set of agents $\Agents$; each of them associated with an individual convex cost function $\vwup[\obj]\from \vecspace\to \R$. At each time $\run=\running$, only a (finite) subset of agents $\current[\Agents]$ is active and may communicate using undirected communications links $\current[\Edges] \defeq \{ \{\agent,\agentalt\}\in\current[\Agents]^2 : \agent \text{ and } \agentalt \text{ can exchange at } \run \}$.
Let $\vt[\nWorkers]\defeq\card(\vt[\Agents])$ denote the number of active agents at time $\run$ and we will also write $\vt[\nWorkersSum]\defeq\sum_{\runalt=\start}^{\run}\vt[\nWorkers][\runalt]$.



In open multi-agent systems, it is in general impossible to define a temporally invariant global objective to minimize over time.
Since those who are present in the network are usually the entities that we really care about, an alternative is to focus exclusively on the active agents and define the instantaneous loss at time $\run$ as
\begin{align}
    \notag
        \vt[\obj^{\text{inst}}](\point)
        =\frac{1}{\vt[\nWorkers]}\sum_{\agent\in\current[\Agents]} \vwup[\obj](\point).
\end{align}
The problem of interest is then the minimization of $\vt[\obj^{\text{inst}}]$.
However, providing a proper analysis for this time-varying problem is still challenging because
the set of active agents $\current[\Agents]$ may change drastically between two consecutive time instants, which also leads to an abrupt change in the objective. In addition, agreeing on a consensus value in an open network is already a difficult problem \cite{franceschelli2020stability, de2020open}.

\subsection{Quantity of Interest}

Instead of tackling the minimization of $\vt[\obj^{\text{inst}}]$ directly, we draw inspiration from online learning and analyze the \emph{running loss} defined for a time-horizon $\nRuns$ as 
\begin{align}
    \label{eq:running-loss}
    \runloss(\nRuns)\!=\!\frac{1}{\vt[\nWorkersSum][\nRuns]}
    \!\!\left(
      \sum_{\run=\start}^{\nRuns}\!\sum_{\worker\in\current[\Agents]} \vwup[\obj](\vt[\state^{\text{ref}}])
    - \min_{\comp\in\points} \!\sum_{\run=\start}^{\nRuns}\!\sum_{\worker\in\current[\Agents]} \!\vwup[\obj](\comp)\!
    \right)
\end{align}
where $\points\subset\vecspace$ is the shared constrained set and $\vt[\state^{\text{ref}}]$ is a reference point for time $\run$. In the sequel, we will consider algorithms in which each agent $\worker\in\vt[\Agents]$ produces a local variable $\vwt[\state]$ at time $\run$. It is thus reasonable to set $\vt[\state^{\text{ref}}]=\vwtATt[\state]$ for a reference agent $\vt[\worker]\in\vt[\Agents]$ that is chosen arbitrarily at each time, and $\runloss$ would then represent the average network suboptimality over time for these reference agents.
Note that the suboptimality is compared with the best fixed {\em a posteriori} action which is the solution $\sol[\comp]\in\points$ of
\begin{align}
    \notag
   \min_{\comp\in\points}
   \left\{
   \vt[\obj^{\text{run}}][\nRuns](\comp)
   \defeq
   \frac{1}{\vt[\nWorkersSum][\nRuns]}
   \sum_{\run=\start}^{\nRuns}\sum_{\worker\in\current[\Agents]} \vwup[\obj](\comp)
   \right\}.
\end{align}

This quantity mimics the collective regret in \cite{hsieh2020multi} with an additional $1/\vt[\nWorkersSum][\nRuns]$ which takes into account the total number of agents that have participated until time $\nRuns$. The notable difference here with the distributed online optimization literature \cite{HCM13,SJ17,YSVQ12} is that we consider an open network of agents so that the composition of the system can change over time.



\section{DAERON: Dual AvERaging for Open Network}
\label{sec:DAERON}
\newcommand{\smallsim}{\smallsym{\mathrel}{\sim}}

\makeatletter
\newcommand{\smallsym}[2]{#1{\mathpalette\make@small@sym{#2}}}
\newcommand{\make@small@sym}[2]{%
  \vcenter{\hbox{$\m@th\downgrade@style#1#2$}}%
}
\newcommand{\downgrade@style}[1]{%
  \ifx#1\displaystyle\scriptstyle\else
    \ifx#1\textstyle\scriptstyle\else
      \scriptscriptstyle
  \fi\fi
}
\makeatother

\begin{algorithm}[tb]
    \caption{DAERON at node $\worker$ for active\;period\;$\jointime$-$\leavetime$}
    \label{algo:DAERON}
\begin{algorithmic}[1]
    \STATE {\bfseries Parameters:}
    $\jointime$ time when the agent joins the network;\\
    $\leavetime$ time when the agent leaves the network
    \STATE {\bfseries Initialize:}
         $\vwt[\set][\worker][\jointime] \subs \vwt[\set][\workeralt][\jointime]$ for $\workeralt\in\vt[\Agents][\jointime-1]\intersect\vt[\Agents][\jointime]$
    \FOR{$\run=\jointime,\ldots,\leavetime$}
    \STATE Generate the prediction $\vwt[\state]$ using \eqref{eq:DDDA}
    \STATE Compute the local subgradient $\vwt[\gvec]\in\subd\vw[\obj](\vwt[\state])$
    \STATE Send subgradients to other active agents
    \STATE Receive subgradients identified by the index set $\vwt[\gvecs]$
    \STATE Update $\vwtupdate[\set]\subs\vwt[\set]\union\vwt[\gvecs]\union\{\vwt[\gvec]\}$
    \ENDFOR
\end{algorithmic}
\end{algorithm}

\subsection{Algorithm}

\ac{DAERON}, as described in \cref{algo:DAERON}, is a (sub)gradient-based method that applies dual averaging \cite{Nes09} at the level of the whole network.
For this, we assume that given any point $\vwt[\point]\in\points$, the agent $\worker$ is able to compute a subgradient $\vwt[\gvec]\in\subd\vw[\obj](\vwt[\point])$.
This information is transmitted to the whole network and used by all the agents for computing their own variable.
Formally, let us define $\vwt[\set]$ as the index set of the subgradients that the agent $\worker\in\vt[\Agents]$ has received/computed by time $\run$.
Then, the agent $\worker$ generates the variable $\vwt[\state]$ as
%
%
%
\begin{equation}
\label{eq:DDDA}
    \vwt[\state] = \proj_{\points} \left(
    \vt[\point][\start]
    -\vwt[\step]\!\sum_{(\workeralt,\runalt)\in\vwt[\set]}\vwt[\gvec][\workeralt][\runalt]
    \right).
\end{equation}
where $\proj_{\points}\from\dvec\mapsto\argmin_{\point\in\points}\norm{\point-\dvec}$ denotes the 
projection onto $\points$, $\vt[\state][\start]$ is a common starting point, 
and $\vwt[\step]>0$ is a learning rate that can be both time and agent dependent.
Note also that the communication of the subgradients (lines 6-7) can be done asynchronously in parallel with the other steps (lines 4-5) of the algorithm.

\begin{remark}[Arriving agents]
In \cref{algo:DAERON}, we initialize an agent arriving at time $\run$ with the subgradient set of an arbitrary agent in $\vt[\Agents][\run-1]\intersect\vt[\Agents]$, which implicitly assumes that $\vt[\Agents][\run-1]\intersect\vt[\Agents]\neq\emptyset$.
This is not crucial, as what really counts is that the agent is initialized with sufficient knowledge about what the network has computed, but we will stick to this model in the remainder of the paper for simplicity.
\end{remark}



\subsection{Practical Implementation}\label{sec:practical}

Storing all the available subgradients at a node can be prohibitively expensive, or even infeasible since this would require infinite storage capacity when $\run$ goes to infinity.
It is thus important to note that \ac{DAERON} is just a conceptual algorithm that can be implemented in different ways to circumvent this issue. 
We provide below two possible workarounds to demonstrate the flexibility of our method:

\begin{itemize}[leftmargin=5mm,itemsep=1mm]
    \item We can maintain $\vwt[\dvec]=\sum_{(\workeralt,\runalt)\in\vwt[\set]}\vwt[\gvec][\workeralt][\runalt]$ while keeping track of the most recent subgradients in order to communicate them with other agents.
    Formally, suppose that each node has a number of potential neighbors they may be connected to; then a subgradient $\vwt[\gvec][\workeralt][\runalt]$ only needs to be stored until the time that it has been sent to or received from these potential neighbors.
    
    \item If the number of involved agents $\nWorkers$ is small, we may define the sum of the subgradients computed by agent $\agent$ as  $\vwt[\agentsum]=\sum_{\runalt=1}^{\run-1}\vwt[\gvec][\worker][\runalt]$ with the convention $\vwt[\gvec][\worker][\runalt]=0$ when $\worker\notin\vt[\Agents][\runalt]$.
    Each node $\agent\in\vt[\Agents]$ then stores a table of size $\nWorkers\times\vdim$ containing the vectors $\vwt[\agentsum][1][\run-\delayij[\run][1]], \ldots, \vwt[\agentsum][\nWorkers][\run-\delayij[\run][\nWorkers]]$ corresponding to delayed versions of the computed subgradient sums over the network, with  $\delayij$ measuring this delay.
    These vectors are updated through communication. This strategy can also be adopted in a semi-centralized network with $\nWorkers$ central nodes and an arbitrary number of edge computing devices that retrieve information from these central nodes.
\end{itemize}


\subsection{Quantities at play and assumptions}

For our analysis, we make the following technical assumptions on the local cost functions and the constraint set.

\vspace{2pt}
\begin{assumption}
\label{asm:cost}
Each $\vwup[\obj]$ is convex and $\gbound$-Lipschitz. The common constraint set $\points$ is closed and convex.
\end{assumption}
\vspace{2pt}

In order for the problem to make sense, we will assume that the communication delays are upper bounded.

\vspace{2pt}
\begin{assumption}
\label{asm:delay}
The delays of the algorithm are upper bounded 
by $\delaybound$, \ie for any $\run,\runalt\in\N$ such that $\run>\runalt+\delaybound$ and any $\worker\in\vt[\Agents]$, $\workeralt\in\vt[\Agents][\runalt]$, we have $(\workeralt,\runalt)\in\vwt[\set]$.
\end{assumption}
\vspace{2pt}

In particular, \cref{asm:delay} supposes that every piece of information is spread to the whole network in a finite amount of time. While this is quite evident in a static network,
in the case of an open network this requires that the evolution of the network is slow enough with respect to the communication between the agents.

For concreteness, let us consider a model where every node communicates all its available gradients to its neighbors at each iteration. Then for a static network with a fixed topology $\graph=(\vertices,\edges)$, we have clearly $\delaybound=\diam(\graph)$ where $\diam(\graph)$ stands for the diameter of the graph.
The case of an open network with changing topology is considerably more complicated. For instance, in the case of a line graph where at each time a new agent joins at the end of the graph, the diameter grows linearly with $\run$ and it becomes impossible to propagate information to the whole network.
This is one kind of situation we want to avoid here.
Formally, we define $\vt[\Agents^{\workeralt,\runalt}]=\setdef{\worker\in\vt[\Agents]}{(\workeralt,\runalt)\in\vwt[\set]}$ as the set of active agents that possess $\vwt[\gvec][\workeralt][\runalt]$ at time $\run$.
The following example provides a sufficient stability assumption which allows information to spread across the network.

\begin{example}[Bounded delays]
\label{ex:vertex-connected}
Come back to the same situation as above where every node communicates all its available gradients to its neighbors during each iteration. Suppose further that the set of the active agents remain unchanged for periods of $\evolveInt$ iterations (\ie for $\run=\runano\evolveInt,\ldots,(\runano+1)\evolveInt-1$), and that the graph $\graph_\runano=(\vt[\vertices],\Union_{\run=\runano\evolveInt}^{\runano(\evolveInt+1)-1}\vt[\edges][\runalt])$ is $\connectivity$-vertex-connected (\ie removing at least $\connectivity$ nodes is necessary to disconnect it).
If the number of arriving plus leaving agents at the end of iteration $(\runano+1)\evolveInt-1$ is $\nChange<\connectivity$, then for any $\runalt\leq\runano\evolveInt-1$ and $\workeralt\in\vt[\Agents][\runalt]$, the number of nodes that do not possess $\vwt[\gvec][\workeralt][\runalt]$ is reduced by at least $\connectivity-\nChange>0$ after the $\evolveInt$ iterations, \ie $\card(\vt[\Agents^{\workeralt,\runalt}][(\runano+1)\evolveInt])\le\max(0,\card(\vt[\Agents^{\workeralt,\runalt}][\runano\evolveInt])-(\connectivity-\nChange))$.
In that case, the maximal delay is thus bounded by $\max_\run \vt[\nWorkers]\evolveInt/( \connectivity-\nChange)+\evolveInt$. 
\end{example}

The above example means that provided that the number of exiting/incoming agents is not too large compared to the connectivity of the communication graph, the delays are naturally bounded.

\begin{remark}[Loss of information due to agent departure]
When an agent leaves the network,
some of its computed subgradients may be lost for the network.
This can happen either because it never communicated them or because the agents to which it communicated them also left the network.
In \cref{ex:vertex-connected}, the change of composition at the end of iteration $(\runano+1)\evolveInt-1$ could notably cause the loss of the subgradients computed at time $\run=\runano\evolveInt,\ldots,(\runano+1)\evolveInt-1$.
This does not affect the proposed algorithm. The only difference is in the analysis: we will consider that the agent $\workeralt$ is not in $\current[\Agents][\runalt]$ if $\vt[\Agents^{\workeralt,\runalt}]=\emptyset$ starting from some $\run$.\footnote{%
Note that in \cref{ex:vertex-connected}, we have either $\vt[\Agents^{\workeralt,\runalt}]=\emptyset$ or $\vt[\Agents^{\workeralt,\runalt}]=\vt[\Agents]$ for $\run$ sufficiently large (assuming that $\runalt$ is fixed).
}
\end{remark}




\section{Analysis}
\label{sec:analysis}
\subsection{Performance in the general case}

We provide the convergence result for \ac{DAERON} in terms of the running loss defined in \eqref{eq:running-loss}. To do so, we define the quadratic mean and the average number of active agents over time:

\begin{equation}
    \notag
    \rms{\nWorkers}=\sqrt{\frac{1}{\nRuns}\sum_{\run=\start}^\nRuns\vt[\nWorkers^2]}
    \quad\text{and}\quad
    \avg[\nWorkers]=\frac{1}{\nRuns}\sum_{\run=\start}^\nRuns\vt[\nWorkers].
\end{equation}
Note that we always have $\rms{\nWorkers}\ge\avg[\nWorkers]$.


\vspace*{1ex}
\begin{theorem}
\label{thm:delay-regret-global}
Let \cref{asm:cost,asm:delay} hold.
Running \ac{DAERON} with constant learning rate
\begin{equation}
    \notag
    \vwt[\step]\equiv
    \step
    =
    \frac{\smallradius_0}{\rms{\nWorkers}\gbound\sqrt{(3\delaybound+1)\nRuns}}
\end{equation}
for some $\smallradius_0>0$ guarantees that 
\begin{equation}
    \label{eq:regret-fixed-lr}
    \runloss(\nRuns)
    \le \frac{\rms{\nWorkers}}{\avg[\nWorkers]}
    \frac{2\regcst\smallradius\gbound\sqrt{3\delaybound+1}}{\sqrt{\nRuns}},
\end{equation}
where $\smallradius=\norm{\sol[\comp]-\vt[\state][\start]}^2$ and
$\regcst=\max(\smallradius/(2\smallradius_0), \smallradius_0/\smallradius)$.
\end{theorem}

\begin{proof}
An important feature of \name is that the subgradients are not always applied to the points where they are evaluated.
To accommodate this in our analysis, we consider the virtual iterates $\seqinf[\vt[\virtual]]$ defined by
\begin{equation}
    \notag
    \vt[\virtual] = \proj_{\points}\left(\vt[\state][\start]-
    \step\sum_{\run=\start}^{\run-1}\sum_{\worker\in\vt[\Agents]}\vwt[\gvec]\right).
\end{equation}

From the regret analysis of dual averaging (see, \eg \cite[Prop. 2]{DHS11}), the following holds for all $\comp\in\points$,
\begin{equation}
    \notag
    \begin{aligned}[b]
    \sum_{\run=1}\sum_{\worker\in\vt[\Agents]}
    \product{\vwt[\gvec]}{\vt[\virtual]-\comp}
    &\le
    \frac{\norm{\comp-\vt[\state][\start]}^2}{2\step}
    +\frac{\step}{2}\sum_{\run=1}^{\nRuns}
    \bigg\|\sum_{\worker\in\vt[\Agents]}\vwt[\gvec]\bigg\|^2\\
    &\le
    \frac{\norm{\comp-\vt[\state][\start]}^2}{2\step}
    +\frac{\step}{2}\sum_{\run=1}^{\nRuns}\vt[\nWorkers^2]\gbound^2.
    \end{aligned}
\end{equation}
In the second line we have used the Lipschitz continuity of the functions which implies that $\norm{\vwt[\gvec]}\le\gbound$.
Next, by using the convexity and the Lipschitz continuity of the functions, each term of \eqref{eq:running-loss} can be bounded for any $\comp\in\points$ by
\begin{align}
    \notag
    \vw[\obj](\vt[\state^{\text{ref}}]) - \vw[\obj](\comp)
    &= \vw[\obj](\vwtATt[\state]) - \vw[\obj](\vwt[\state])
    + \vw[\obj](\vwt[\state]) - \vw[\obj](\comp)\\
    \notag
    &\le \gbound\norm{\vwtATt[\state]-\vwt[\state]}
    + \product{\vwt[\gvec]}{\vwt[\state]-\comp}.
\end{align}
We proceed to bound the second term in the above inequality
\begin{equation}
    \notag
    \begin{aligned}[b]
    \product{\vwt[\gvec]}{\vwt[\state]-\comp}
    &= \product{\vwt[\gvec]}{\vwt[\state]-\vt[\virtual]}
    + \product{\vwt[\gvec]}{\vt[\virtual]-\comp}\\
    &\le \gbound\norm{\vwt[\state]-\vt[\virtual]}
    + \product{\vwt[\gvec]}{\vt[\virtual]-\comp}.
    \end{aligned}
\end{equation}
Using the triangle inequality, we get $\norm{\vwtATt[\state]-\vwt[\state]}\le\norm{\vwtATt[\state]-\vt[\virtual]}+\norm{\vwt[\state]-\vt[\virtual]}$.
Now, let $\vt[\Gamma] = \max_{\worker\in\vt[\Agents]}\norm{\vwt[\state]-\vt[\virtual]}$ and set $\comp\subs\sol[\comp]$. We have from the above
\begin{equation}
    \notag
    \runloss(\nRuns)
    \le \frac{1}{\vt[\nWorkersSum]}
    \left(
    \frac{\norm{\sol[\comp]-\vt[\state][\start]}^2}{2\step}
    +\sum_{\run=1}^{\nRuns}
    (\frac{\step}{2}\vt[\nWorkers^2]\gbound^2
    +3\vt[\nWorkers]\gbound\vt[\Gamma])
    \right).
\end{equation}

To conclude, we may bound $\vt[\Gamma]$ thanks to the bounded delay assumption.
In fact, since the projection operator is non-expansive, it holds for all $\worker\in\vt[\Agents]$ that 
\begin{equation}
    \notag
    \addtocounter{footnote}{-1}
    \begin{multlined}
    \norm{\vwt[\state]-\vt[\virtual]}
    \le \bigg\|
    \step\sum_{\run=\start}^{\run-1}
    \sum_{\workeralt\in\vt[\Agents]}\vwt[\gvec]
    - 
    \step\sum_{(\workeralt,\runalt)\in\vwt[\set]}\vwt[\gvec][\workeralt][\runalt]
    \bigg\| \\
    \le \bigg\|
    \step\sum_{\runalt=\run-\delaybound}^{\run-1}
    \sum_{\workeralt\in\vt[\Agents][\runalt]}\vwt[\gvec][\workeralt][\runalt]
    \bigg\|
    \le \step\sum_{\runalt=\run-\delaybound}^{\run-1} \vt[\nWorkers][\runalt]\gbound.\footnotemark
    \end{multlined}
    \footnotetext{If $\runalt\le0$, then $\vt[\Agents][\runalt]=\emptyset$ and $\vt[\nWorkers]=0$.}
\end{equation}
Then, with $\vt[\nWorkers]\vt[\nWorkers][\runalt]\le(\vt[\nWorkers^2]+\vt[\nWorkers^2][\runalt])/2$, we obtain the following
\begin{equation}
    \notag
    \runloss(\nRuns)
    \le \frac{1}{\vt[\nWorkersSum]}
    \left(
    \frac{\norm{\sol[\comp]-\vt[\state][\start]}^2}{2\step}
    +\step\gbound^2\sum_{\run=1}^{\nRuns}
    (3\delaybound+1)\vt[\nWorkers^2]
    \right).
\end{equation}
Inequality \eqref{eq:regret-fixed-lr} follows immediately by the definition of $\rms{\nWorkers}$, $\avg[\nWorkers]$, and $\regcst$.
\end{proof}

\cref{thm:delay-regret-global} shows that when the ratio $\rms{\nWorkers}/\avg[\nWorkers]$ is upper bounded, the running loss of the algorithm has a convergence rate in $\bigoh(1/\sqrt{\nRuns})$.
The factor $\rms{\nWorkers}/\avg[\nWorkers]$ also indicates that the algorithm converges slower when the number of active agents varies greatly across iterations,
which is expected because the algorithm would need more time to accommodate the change in this situation.

\begin{remark}[Extensions]
One drawback of the theorem is that the expression of the learning rate involves both the sqaure mean number of the agents $\rms{\nWorkers}$ and the time horizon $\nRuns$.
In a continuously evolving network, neither of these two quantities are known in advance.
We provide below several alternatives which allow us to establish the same $\bigoh(1/\sqrt{\nRuns})$ rate without these quantities. Proofs are variations of the above proof; we omit details due to space limitations. 
\begin{itemize}[leftmargin=5mm,itemsep=1mm]
    \item If the number $\vt[\nWorkers]$ is known to every agent.
    We may use $\vwt[\step]=\vt[\step]=\Theta(1/\sqrt{\delaybound(\sum_{\runalt=1}^{\run}\vt[\nWorkers^2][\runalt])})$.
    \item If $\nWorkers_{\max}=\max_{1\le\run\le\nRuns}\vt[\nWorkers]$ can be estimated and the agents have access to a global clock that indicates $\run$, we can take $\vwt[\step]=\vt[\step]=\Theta(1/(\nWorkers_{\max}\sqrt{\delaybound\run}))$.
    \item Note that $\sqrt{\sum_{\runalt=1}^{\run}\vt[\nWorkers^2][\runalt]}\le\sqrt{\nWorkers_{\max}}\sqrt{\vt[\nWorkersSum]}$.
    Therefore, provided that $\nWorkers_{\max}$ is known, another alternative is to use $\vwt[\step]=\Theta(1/\sqrt{\delaybound\nWorkers_{\max}\vt[\nWorkersSum]})$. This does not require to know $\vt[\nWorkers]$ explicitly since $\vt[\nWorkersSum]$ can be estimated by $\card(\vwt[\set])$. In particular, under \cref{asm:delay}, it holds $\vt[\nWorkersSum]\le\card(\vwt[\set])+(\delaybound+1)\nWorkers_{\max}$.
\medskip
\end{itemize}
\end{remark}

\subsection{The case of fixed agents $\current[\Agents]=\Agents$}

For comparison with existing literature (\eg \cite{duchi2011dual}), we now turn back to the case of a ``closed network'' where all the agents are active at each iteration.
In this situation, we can define the (fixed) global loss as
\begin{equation}
    \notag
    \obj(\point) = \frac{1}{\nWorkers} \sum_{\worker\in\Agents} \vwup[\obj](\point),
\end{equation}
where $\nWorkers = \card(\Agents)$ is the number of agents. We have $\rms{\nWorkers}=\avg[\nWorkers]=\nWorkers$ and thus $\rms{\nWorkers}/\avg[\nWorkers]=1$.
As an immediate corollary of \cref{thm:delay-regret-global}, if we apply \ac{DAERON} with the constant stepsize
\[\vwt[\step]\equiv\step=\frac{\smallradius}{\nWorkers\gbound\sqrt{(3\delaybound+1)\nRuns}},\]
then for any  agent $\worker\in\Agents$, 
%
\begin{align*}
 \obj
 \left( \frac{1}{\nRuns}\sum_{\run=\start}^{\nRuns}\vwt[\state]\right)
    - \min_{\comp\in\points} \obj(\comp) &\le
\frac{1}{\nRuns}\sum_{\run=\start}^{\nRuns} \obj(\vwt[\state])
    - \min_{\comp\in\points} \obj(\comp) \\
    &\le \frac{    2\smallradius\gbound\sqrt{3\delaybound+1}}{\sqrt{\nRuns}}
.
\end{align*}
This means
that the running average of each agent's iterates decreases in global suboptimality at rate $\bigoh(1/\sqrt{\nRuns})$,
which matches the rate of \cite[Th.~2]{duchi2011dual} for the slightly different decentralized dual averaging algorithm.\footnote{In \cite{duchi2011dual}, the subgradient are averaged by gossiping while in  \ac{DAERON}, they are directly exchanged.}
Going one step further, the same result would still hold under asynchronous activation as long as the activation patterns can be described by a stationary probability distribution. 
On the other hand, \cref{thm:delay-regret-global} also allows us to prove the convergence of the algorithm in an open network whose composition stops changing after finite time.

\section{Numerical Experiments}
\label{sec:experiments}
In this section, we demonstrate the effectiveness of \ac{DAERON} with experiments on a static and an open network.

\subsection{Problem Description}

Let us consider a decentralized \ac{LAD} regression model.
Given a data set evenly distributed on $\nWorkers$ nodes $(\expvar_{\worker\indg},\resvar_{\worker\indg})_{\worker,\indg\in\oneto{\nWorkers}\times\oneto{\nSamplesPerWorker}}$ with $\expvar_{\worker\indg}$ in $\vecspace$ and $\resvar_{\worker\indg}\in\R$, it consists in solving 
\begin{equation}
    \label{eq:LAD}
    \min_{\point\in\vecspace}
    \left\{\obj(\point)\defeq
    \frac{1}{\nWorkers}\sumworker\frac{1}{\nSamplesPerWorker}\sum_{\indg=1}^\nSamplesPerWorker \abs{\expvar_{\worker\indg}^{\top}\point-\resvar_{\worker\indg}}\right\}.
\end{equation}
Compared to least square regression, \ac{LAD} is known to be more resistant to the presence of outliers. 
Although the use of absolute value makes the problem non-differentiable, \ac{DAERON} can be run with subgradients as suggested by our analysis.
For the experiments, we generate synthetic data as follows:
\begin{enumerate}
    \item The ground truth model $\groundtruth\in[-5,5]^{\vdim}$ is drawn from a uniform distribution.
    \item The local model $\vw[\groundtruth]$ of the node $\worker$ is obtained by perturbing $\groundtruth$ with a Gaussian noise, \ie $\vw[\groundtruth]=\groundtruth+\vw[\varepsilon]$ where $\vw[\epsilon]\sim\mathcal{N}(0,\Id_{\vdim})$.
    \item We sample $\expvar_{\worker\indg}\sim\mathcal{N}(0,\Id_{\vdim})$ and compute $\resvar_{\worker\indg} = \expvar_{\worker\indg}^{\top}\vw[\groundtruth]+\varepsilon_{\worker\indg}$ with $\varepsilon_{\worker\indg}\sim\mathcal{N}(0,1)$.
    \item On each node, a random portion of samples are corrupted. For these samples, we replace $\resvar_{\worker\indg}$ by a random value generated from a Gaussian distribution.
\end{enumerate}
In the above, we introduce the second and the fourth steps mainly for two reasons. First, it makes the problem more heterogeneous, and thus more difficult. Second, it makes the communication between agents more important for finding a good approximation of $\groundtruth$.
In the following, we will take $\nWorkers=64$ nodes, $\nSamplesPerWorker=200$ samples per node, and dimension  $\vdim=20$.
On each node, the number of corrupted samples is random in $\intinterval{0}{120}$. We also verify that the solution $\expsol$ of \eqref{eq:LAD} is not too far from $\groundtruth$.

\begin{algorithm}[tb]
    \caption{DAERON at each node $\worker$ as implemented in \cref{subsec:exp-open}}
    \label{algo:open-network}
\begin{algorithmic}[1]
    \STATE {\bfseries Initialize:}
         $\vwt[\set][\worker][\start] \subs \emptyset$, activation status $\vw[\activeStatus]\in\{0,1\}$, network parameters $K\in\N$, $p\in[0,1]$
    \FOR{$\run=\running$}
    \STATE \texttt{\underline{Agent update}}\\[0.1em]
    \IF{$\vw[\activeStatus]=1$}
    \STATE Get randomly paired with another active agent $\workeralt$
    \STATE Compute $\vwt[\state]$ by \eqref{eq:DDDA} and $\vwt[\gvec]\in\subd \vw[\obj](\vwt[\state])$
    \STATE Update $\vwtupdate[\set]\subs\vwt[\set]\union\vwt[\set][\workeralt]\union\{\vwt[\gvec]\}$
    \ENDIF
    \STATE \texttt{\underline{Network evolution}}\\[0.1em]
    \IF{$(\run+1)\equiv 0\mod K$}
    \STATE Draw a Bernoulli random variable $\vw[z]\sim\mathcal{B}(p)$
    \STATE $\vw[\activeStatus] \subs \vw[\activeStatus]+\vw[z]\mod 2$
    \IF{$\vw[z]=1$ and $\vw[\activeStatus] = 1$}
    \STATE Pick randomly $\workeralt \in \vt[\Agents]\intersect\vt[\Agents][\run+1]$
    \STATE Update $\vwtupdate[\set]\subs\vwtupdate[\set][\workeralt]$
    \ENDIF
    \ENDIF
    \ENDFOR
\end{algorithmic}
\end{algorithm}

 

\subsection{Static network}

We first investigate the performance of the algorithm on a static network.
The nodes are arranged in a $2$d grid of size $8\times8$.
Adjacent nodes exchange gradients at each iteration.
Communication-computation overlap is allowed for better efficiency--- in \cref{algo:DAERON}, this means that lines 4-5 and 6-7 are run in parallel.
Then, with a constant stepsize $\step$, the update writes
\begin{equation}
    \notag
    \vwtupdate[\state] = \vwt[\state] - \step \sumworker \vwt[\gvec][\workeralt][\run-\delay_{\workeralt,\worker}],
\end{equation}
where $\delay_{\workeralt,\worker}$ is the distance between the nodes $\workeralt$ and $\worker$.
For illustration purposes, we also compare with a \ac{DGD} method \cite{nedic2009distributed} with constant stepsize\;$\stepalt$ and a mixing matrix $\mixing=(\mixingw_{\worker,\workeralt})$. Its update is
\begin{equation}
    \notag
    \vwtupdate[\state] = \sum_{\workeralt=1}^{\nWorkers}\mixingw_{\worker,\workeralt}\vwt[\state][\workeralt] - \stepalt \vwt[\gvec]
\end{equation}
and we take $\mixing$ as the Metropolis matrix of the graph in our experiments:
\begin{equation}
    \notag
    \mixingw_{\worker,\workeralt}
    = \left\{
        \begin{array}{ll}
            1/(\max(\degr(\worker),\degr(\workeralt))+1) & \text{if } \{\worker,\workeralt\}\in\edges,\\
            1 - \sum_{\indg=1}^{\nWorkers}\mixingw_{\worker,\indg} & \text{if } \worker=\workeralt,\\
            0 & \text{otherwise},
        \end{array}
        \right.
\end{equation}
with $\degr(\worker)$ denoting the degree of the node $\worker$.

For a proper comparison of the two algorithms, it is important to notice that a subgradient is sent to all the $\nWorkers$ nodes in \name while it is averaged out in \ac{DGD}.
Therefore, we will take $\stepalt=\nWorkers\step$ and refer to it as the \emph{effective learning rate} of both methods.
With this in mind, in \cref{subfig:convergence-static-network} we plot the convergence of the averaged optimality gap $(1/\nWorkers)\sumworker\obj(\vwt[\state])-\min\obj$ for \ac{DAERON} and \ac{DGD} with different choices of $\stepalt$.

Interestingly, we observe that when using the same effective learning rate, the two algorithms establish similar convergence behavior until reaching their respective fixed points. However, \ac{DAERON} is able to converge to a point with higher accuracy.
We believe that this is because the variables $(\vwt[\state])_{\worker\in\Agents}$ tend to be closer to each other in \ac{DAERON}.\footnote{Indeed, let $\Delta$ and $1-\lambda$ be respectively the diameter of the graph and the spectral gap of the mixing matrix.
Then, as shown in the proof of \cref{thm:delay-regret-global}, for \ac{DAERON} we can roughly bound $\norm{\vwt[\state]-\vwt[\state][\workeralt]}$ by $\nWorkers\step\delaybound\gbound=\Delta\stepalt\gbound$.
As for \ac{DGD}, it is known that we have asymptotically $\norm{\vwt[\state]-\vwt[\state][\workeralt]}\lesssim\gamma\gbound/(1-\lambda)$ \cite[Lem.~11]{NOR18}.
Since it generally holds $\Delta\le1/(1-\lambda)$ and this is notably the case for the graph that we consider, this provides a possible explanation for the superiority of \ac{DAERON} over \ac{DGD}.}


\begin{figure}
    \centering
    \begin{subfigure}[b]{0.49\linewidth}
    \centering
    \includegraphics[width=\textwidth]{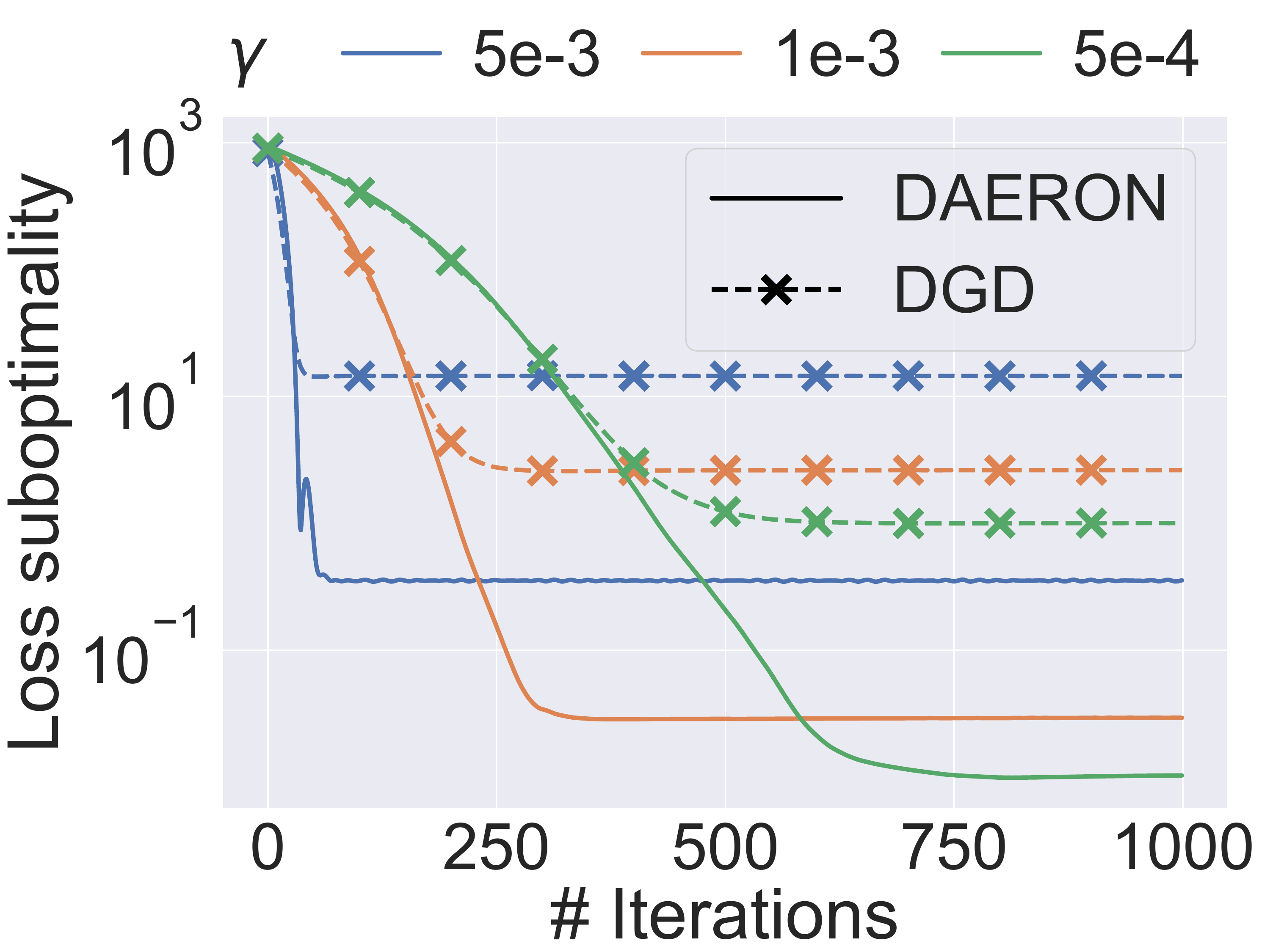}
    \caption{Static network}
    \label{subfig:convergence-static-network}
    \end{subfigure}
    \hfill
    \begin{subfigure}[b]{0.49\linewidth}
    \centering
    \includegraphics[width=\textwidth]{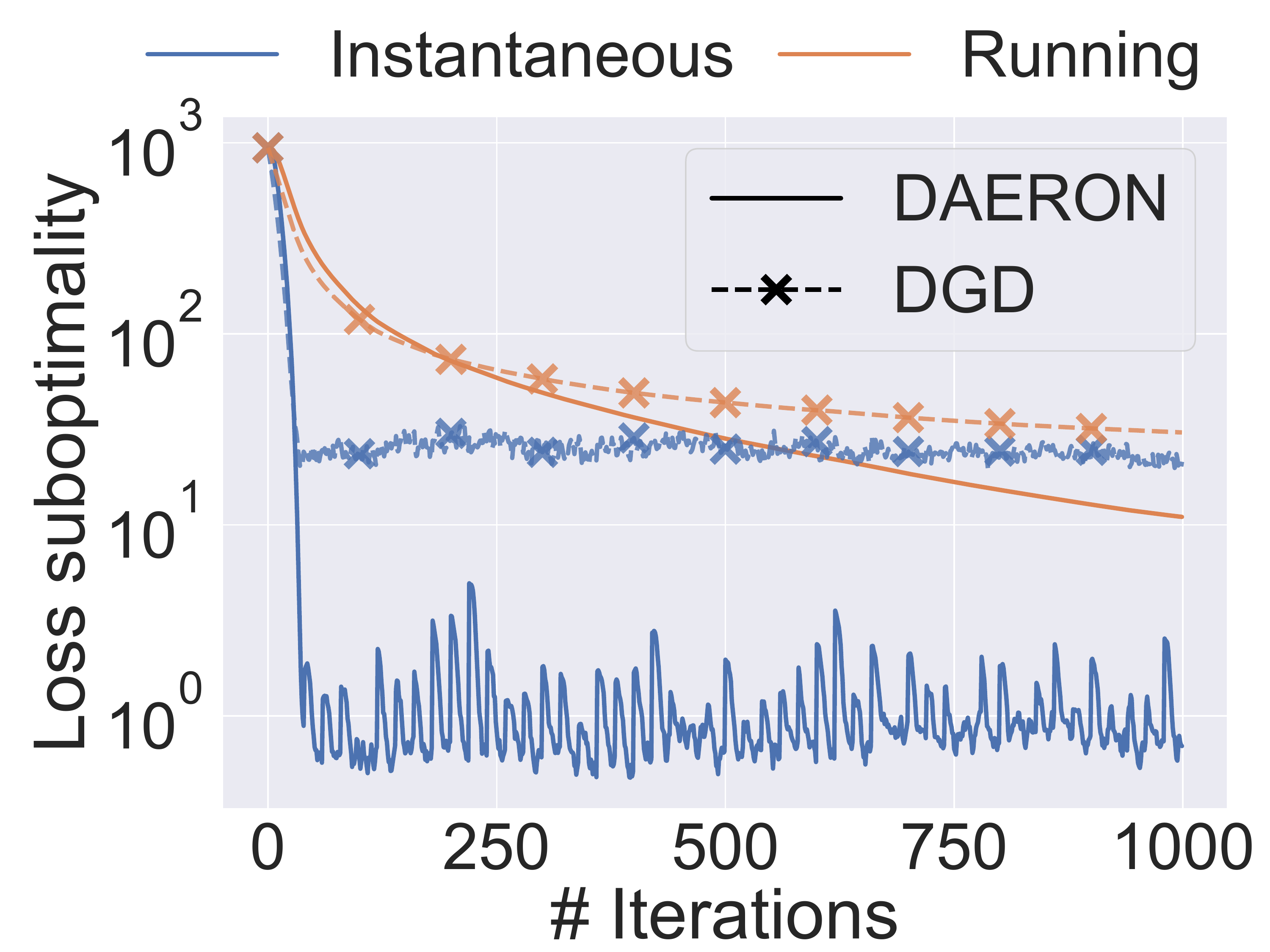}
    \caption{Open network}
    \label{subfig:convergence-open-network}
    \end{subfigure}
    \caption{Comparison of \ac{DAERON} and \ac{DGD}.
    For a static network we plot in (a) the averaged suboptimality. For an open network we plot in (b)
    the averaged instantaneous suboptimality \eqref{eq:avg-inst-loss} and the averaged running loss \eqref{eq:avg-run-loss}.}
    \label{fig:convergence}
\end{figure}

\subsection{Open network}
\label{subsec:exp-open}

Now that we have shown that \ac{DAERON} performs comparably to standard decentralized optimization methods in a static network, we proceed to study its behavior in an open multi-agent system (\cref{algo:open-network}).
Following \cite{VMN18}, we model the arrivals and departures of the agents by a Bernoulli process.
Initially, only half of the $64$ nodes are active. Then, every $K=20$ iterations, each agent may change its activation status (\ie from active to inactive or vice-versa) with probability $p=0.05$.
At each iteration, the active nodes are randomly paired with each other, then each pair synchronizes their gradients.\footnote{%
If there is an odd number of nodes, one node is ignored in this process.}
Formally, if nodes $\worker$ and $\workeralt$ are paired at time $\run$, then 
%
$
    \setexclude{\vwtupdate[\set]}{\{\vwt[\gvec]\}} = \setexclude{\vwtupdate[\set][\workeralt]}{\{\vwt[\gvec][\workeralt]\}}
    =\vwt[\set]\union\vwt[\set][\workeralt].
$
%
In the spirit of \ac{DGD}, we also implement an algorithm which directly updates the primal variables as
%
\begin{equation}
    \notag
    \begin{aligned}
    \vwtupdate[\state]&=\frac{\vwt[\state]+\vwt[\state][\workeralt]}{2}-\stepalt\vwt[\gvec],\\
    \vwtupdate[\state][\workeralt]&=\frac{\vwt[\state]+\vwt[\state][\workeralt]}{2}-\stepalt\vwt[\gvec][\workeralt].
    \end{aligned}
\end{equation}
In both cases, an agent that becomes active at the end of round $\run-1$ is assigned the state variable (\ie $\vwt[\set]$ or $\vwt[\state]$) of a random node in $\vt[\Agents][\run-1]\intersect\vt[\Agents]$.
We take $\stepalt=\nWorkers\step/2$ in our experiment since on average $\nWorkers/2$ nodes are active.

As for the performance measure, we consider the averaged instantaneous optimality gap $\avgInstLoss(\run)$ and the averaged running loss $\avgRunLoss(\run)$ defined by
%
\begin{align}
\avgInstLoss(\run)
&=
\frac{1}{\vt[\nWorkers]}\sum_{\worker\in\vt[\Agents]}\vt[\obj^{\text{inst}}](\vwt[\state])-\min\vtp[\obj]
\label{eq:avg-inst-loss}
\\
\avgRunLoss(\run)
&=
\frac{1}{\vt[\nWorkersSum]}
\sum_{\runalt=1}^{\run}
\sum_{\worker\in\vt[\Agents][\runalt]}\vt[\obj^{\text{inst}}][\runalt](\vwt[\state][\worker][\runalt])
-\min\vt[\obj^{\text{run}}].
\label{eq:avg-run-loss}
\end{align}
The evolution of these two measures for $\stepalt=0.005$ are plotted in \cref{subfig:convergence-open-network}.\footnote{This is roughly the largest stepsize that leads to the decrease of the losses.
We observe similar convergence patterns for smaller stepsizes, while the performance gap between the two algorithm diminishes.}
We see that both algorithms are able to converge to an area where potential solutions are located whereas \ac{DAERON} can get much closer to the optimum of $\vt[\instObj]$.
Moreover, we observe that the instantaneous loss for \ac{DAERON} increases abruptly when the set of active agents changes and gets decreased again afterwards.
This indicates that the algorithm actually has the ability to track the instantaneous solution.


\section{Conclusion}
\label{sec:conclusion}
In this paper, we introduced a decentralized optimization method for open multi-agent systems, in which agents can freely join or leave the network during the process.
This setting is particularly challenging since there is not a clear objective to minimize and even exact convergence to a consensus is generally out of reach.
These difficulties led us to adopting techniques and performance measures from online optimization.
We proved that our algorithm benefits from a $\bigoh(1/\sqrt{\nRuns})$ convergence rate with respect to the running average of the functions of the agents present in the network. 
In our simulations, we showed that our method outperformed decentralized subgradient descent on a least absolute deviation problem.

The positive results in our numerical experiments also opens up several new research directions: To analyze the tracking behavior, it could be more relevant to focus on the dynamic regret \cite{HW15,SJ17} or to adopt the perspective of time-varying optimization \cite{Simonetto17,DSSM20}.
The goal here is either to bound the sum of the instantaneous optimality gaps or to show that the realized actions is always close to the current optimum.
Another promising direction is to explore the potential combination of \ac{DAERON} and consensus-based methods, which may lead to superior performance.
Given the variety of problems and open networks that we are facing, it is clear that there is not a single method that would outperform all the others. It is thus also important to find the most suitable algorithm for each specific setup.




\section*{Acknowledgments}
This work has been partially supported by MIAI Grenoble Alpes (ANR-19-P3IA-0003).

\bibliographystyle{IEEEtran}
\bibliography{IEEEabrv,references}

\begin{thebibliography}{10}
\providecommand{\url}[1]{#1}
\csname url@rmstyle\endcsname
\providecommand{\newblock}{\relax}
\providecommand{\bibinfo}[2]{#2}
\providecommand\BIBentrySTDinterwordspacing{\spaceskip=0pt\relax}
\providecommand\BIBentryALTinterwordstretchfactor{4}
\providecommand\BIBentryALTinterwordspacing{\spaceskip=\fontdimen2\font plus
\BIBentryALTinterwordstretchfactor\fontdimen3\font minus
  \fontdimen4\font\relax}
\providecommand\BIBforeignlanguage[2]{{%
\expandafter\ifx\csname l@#1\endcsname\relax
\typeout{** WARNING: IEEEtran.bst: No hyphenation pattern has been}%
\typeout{** loaded for the language `#1'. Using the pattern for}%
\typeout{** the default language instead.}%
\else
\language=\csname l@#1\endcsname
\fi
#2}}

\bibitem{tsitsiklis1984problems}
J.~N. Tsitsiklis, ``Problems in decentralized decision making and
  computation.'' Massachusetts Inst of Tech Cambridge Lab for Information and
  Decision Systems, Tech. Rep., 1984.

\bibitem{boyd2006randomized}
S.~Boyd, A.~Ghosh, B.~Prabhakar, and D.~Shah, ``Randomized gossip algorithms,''
  \emph{IEEE transactions on information theory}, vol.~52, no.~6, pp.
  2508--2530, 2006.

\bibitem{dimakis2010gossip}
A.~G. Dimakis, S.~Kar, J.~M. Moura, M.~G. Rabbat, and A.~Scaglione, ``Gossip
  algorithms for distributed signal processing,'' \emph{Proceedings of the
  IEEE}, vol.~98, no.~11, pp. 1847--1864, 2010.

\bibitem{lian2018asynchronous}
X.~Lian, W.~Zhang, C.~Zhang, and J.~Liu, ``Asynchronous decentralized parallel
  stochastic gradient descent,'' in \emph{International Conference on Machine
  Learning}.\hskip 1em plus 0.5em minus 0.4em\relax PMLR, 2018, pp. 3043--3052.

\bibitem{nedic2009distributed}
A.~Nedic and A.~Ozdaglar, ``Distributed subgradient methods for multi-agent
  optimization,'' \emph{IEEE Transactions on Automatic Control}, vol.~54,
  no.~1, pp. 48--61, 2009.

\bibitem{bianchi2012convergence}
P.~Bianchi and J.~Jakubowicz, ``Convergence of a multi-agent projected
  stochastic gradient algorithm for non-convex optimization,'' \emph{IEEE
  transactions on automatic control}, vol.~58, no.~2, pp. 391--405, 2012.

\bibitem{jakovetic2014fast}
D.~Jakoveti{\'c}, J.~Xavier, and J.~M. Moura, ``Fast distributed gradient
  methods,'' \emph{IEEE Transactions on Automatic Control}, vol.~59, no.~5, pp.
  1131--1146, 2014.

\bibitem{duchi2011dual}
J.~C. Duchi, A.~Agarwal, and M.~J. Wainwright, ``Dual averaging for distributed
  optimization: Convergence analysis and network scaling,'' \emph{IEEE
  Transactions on Automatic control}, vol.~57, no.~3, pp. 592--606, 2011.

\bibitem{tsianos2012push}
K.~I. Tsianos, S.~Lawlor, and M.~G. Rabbat, ``Push-sum distributed dual
  averaging for convex optimization,'' in \emph{2012 ieee 51st ieee conference
  on decision and control (cdc)}.\hskip 1em plus 0.5em minus 0.4em\relax IEEE,
  2012, pp. 5453--5458.

\bibitem{lee2017stochastic}
S.~{Lee}, A.~{Nedić}, and M.~{Raginsky}, ``Stochastic dual averaging for
  decentralized online optimization on time-varying communication graphs,''
  \emph{IEEE Transactions on Automatic Control}, vol.~62, no.~12, pp.
  6407--6414, 2017.

\bibitem{colin2016gossip}
I.~Colin, A.~Bellet, J.~Salmon, and S.~Cl{\'e}men{\c{c}}on, ``Gossip dual
  averaging for decentralized optimization of pairwise functions,'' in
  \emph{International Conference on Machine Learning}.\hskip 1em plus 0.5em
  minus 0.4em\relax PMLR, 2016, pp. 1388--1396.

\bibitem{NOR18}
A.~Nedi{\'c}, A.~Olshevsky, and M.~G. Rabbat, ``Network topology and
  communication-computation tradeoffs in decentralized optimization,''
  \emph{Proceedings of the IEEE}, vol. 106, no.~5, pp. 953--976, 2018.

\bibitem{boyd2011distributed}
S.~Boyd, N.~Parikh, and E.~Chu, \emph{Distributed optimization and statistical
  learning via the alternating direction method of multipliers}.\hskip 1em plus
  0.5em minus 0.4em\relax Now Publishers Inc, 2011.

\bibitem{iutzeler2013asynchronous}
F.~Iutzeler, P.~Bianchi, P.~Ciblat, and W.~Hachem, ``Asynchronous distributed
  optimization using a randomized alternating direction method of
  multipliers,'' in \emph{52nd IEEE conference on decision and control}.\hskip
  1em plus 0.5em minus 0.4em\relax IEEE, 2013, pp. 3671--3676.

\bibitem{wei20131}
E.~Wei and A.~Ozdaglar, ``On the o (1= k) convergence of asynchronous
  distributed alternating direction method of multipliers,'' in \emph{2013 IEEE
  Global Conference on Signal and Information Processing}.\hskip 1em plus 0.5em
  minus 0.4em\relax IEEE, 2013, pp. 551--554.

\bibitem{Franceschelli2018proportional}
M.~{Franceschelli} and P.~{Frasca}, ``Proportional dynamic consensus in open
  multi-agent systems,'' in \emph{2018 IEEE Conference on Decision and Control
  (CDC)}, 2018, pp. 900--905.

\bibitem{Dashti2019dynamic}
Z.~A.~Z. {Sanai Dashti}, C.~{Seatzu}, and M.~{Franceschelli}, ``Dynamic
  consensus on the median value in open multi-agent systems,'' in \emph{2019
  IEEE 58th Conference on Decision and Control (CDC)}, 2019, pp. 3691--3697.

\bibitem{franceschelli2020stability}
M.~Franceschelli and P.~Frasca, ``Stability of open multi-agent systems and
  applications to dynamic consensus,'' \emph{IEEE Transactions on Automatic
  Control}, 2020.

\bibitem{de2019lower}
C.~M. {de Galland} and J.~M. {Hendrickx}, ``Lower bound performances for
  average consensus in open multi-agent systems,'' in \emph{2019 IEEE 58th
  Conference on Decision and Control (CDC)}, 2019, pp. 7429--7434.

\bibitem{de2020open}
C.~M. de~Galland, S.~Martin, and J.~M. Hendrickx, ``Open multi-agent systems
  with variable size: the case of gossiping,'' \emph{arXiv preprint
  arXiv:2009.02970}, 2020.

\bibitem{hendrickx2020stability}
J.~M. {Hendrickx} and M.~G. {Rabbat}, ``Stability of decentralized gradient
  descent in open multi-agent systems,'' in \emph{2020 59th IEEE Conference on
  Decision and Control (CDC)}, 2020, pp. 4885--4890.

\bibitem{tsianos2012consensus}
K.~I. Tsianos, S.~Lawlor, and M.~G. Rabbat, ``Consensus-based distributed
  optimization: Practical issues and applications in large-scale machine
  learning,'' in \emph{2012 50th annual allerton conference on communication,
  control, and computing (allerton)}.\hskip 1em plus 0.5em minus 0.4em\relax
  IEEE, 2012, pp. 1543--1550.

\bibitem{wu2017decentralized}
T.~Wu, K.~Yuan, Q.~Ling, W.~Yin, and A.~H. Sayed, ``Decentralized consensus
  optimization with asynchrony and delays,'' \emph{IEEE Transactions on Signal
  and Information Processing over Networks}, vol.~4, no.~2, pp. 293--307, 2017.

\bibitem{zhou2018distributed}
Z.~Zhou, P.~Mertikopoulos, N.~Bambos, P.~Glynn, Y.~Ye, L.-J. Li, and
  L.~Fei-Fei, ``Distributed asynchronous optimization with unbounded delays:
  How slow can you go?'' in \emph{International Conference on Machine
  Learning}.\hskip 1em plus 0.5em minus 0.4em\relax PMLR, 2018, pp. 5970--5979.

\bibitem{BKR19}
A.~S. Bedi, A.~Koppel, and K.~Rajawat, ``Asynchronous online learning in
  multi-agent systems with proximity constraints,'' \emph{IEEE Transactions on
  Signal and Information Processing over Networks}, vol.~5, no.~3, pp.
  479--494, 2019.

\bibitem{de2020fundamental}
C.~M. de~Galland and J.~M. Hendrickx, ``Fundamental performance limitations for
  average consensus in open multi-agent systems,'' \emph{arXiv preprint
  arXiv:2004.06533}, 2020.

\bibitem{hsieh2020multi}
Y.-G. Hsieh, F.~Iutzeler, J.~Malick, and P.~Mertikopoulos, ``Multi-agent online
  optimization with delays: Asynchronicity, adaptivity, and optimism,''
  \emph{arXiv preprint arXiv:2012.11579}, 2020.

\bibitem{HCM13}
S.~Hosseini, A.~Chapman, and M.~Mesbahi, ``Online distributed optimization via
  dual averaging,'' in \emph{52nd IEEE Conference on Decision and
  Control}.\hskip 1em plus 0.5em minus 0.4em\relax IEEE, 2013, pp. 1484--1489.

\bibitem{SJ17}
S.~Shahrampour and A.~Jadbabaie, ``Distributed online optimization in dynamic
  environments using mirror descent,'' \emph{IEEE Transactions on Automatic
  Control}, vol.~63, no.~3, pp. 714--725, 2017.

\bibitem{YSVQ12}
F.~Yan, S.~Sundaram, S.~Vishwanathan, and Y.~Qi, ``Distributed autonomous
  online learning: Regrets and intrinsic privacy-preserving properties,''
  \emph{IEEE Transactions on Knowledge and Data Engineering}, vol.~25, no.~11,
  pp. 2483--2493, 2012.

\bibitem{Nes09}
Y.~Nesterov, ``Primal-dual subgradient methods for convex problems,''
  \emph{Mathematical Programming}, vol. 120, no.~1, pp. 221--259, 2009.

\bibitem{DHS11}
J.~Duchi, E.~Hazan, and Y.~Singer, ``Adaptive subgradient methods for online
  learning and stochastic optimization,'' \emph{The Journal of Machine Learning
  Research}, vol.~12, pp. 2121--2159, 2011.

\bibitem{VMN18}
V.~S. Varma, I.-C. Mor{\u{a}}rescu, and D.~Ne{\v{s}}i{\'c}, ``Open multi-agent
  systems with discrete states and stochastic interactions,'' \emph{IEEE
  Control Systems Letters}, vol.~2, no.~3, pp. 375--380, 2018.

\bibitem{HW15}
E.~C. Hall and R.~M. Willett, ``Online convex optimization in dynamic
  environments,'' vol.~9, no.~4, pp. 647--662, February 2015.

\bibitem{Simonetto17}
A.~Simonetto, ``Time-varying convex optimization via time-varying averaged
  operators,'' \emph{arXiv preprint arXiv:1704.07338}, 2017.

\bibitem{DSSM20}
E.~Dall'Anese, A.~Simonetto, S.~Becker, and L.~Madden, ``Optimization and
  learning with information streams: Time-varying algorithms and
  applications,'' \emph{IEEE Signal Processing Magazine}, vol.~37, no.~3, pp.
  71--83, 2020.

\end{thebibliography}

\end{document}